\documentclass[11pt]{amsart}

\usepackage[utf8]{inputenc}
\usepackage[english]{babel}
\usepackage{amsmath,amsfonts,amssymb,latexsym,amsthm}

\makeatletter
\providecommand*{\shuffle}{%
  \mathbin{\mathpalette\shuffle@{}}%
}
\newcommand*{\shuffle@}[2]{%
  \sbox0{$#1\vcenter{}$}%
  \kern .15\ht0 
  \rlap{\vrule height .25\ht0 depth 0pt width 2.5\ht0}%
  \raise.1\ht0\hbox to 2.5\ht0{%
    \vrule height 1.75\ht0 depth -.1\ht0 width .17\ht0 %
    \hfill
    \vrule height 1.75\ht0 depth -.1\ht0 width .17\ht0 %
    \hfill
    \vrule height 1.75\ht0 depth -.1\ht0 width .17\ht0 %
  }%
  \kern .15\ht0 
}
\makeatother

\usepackage{hyperref}
\usepackage{wasysym}
\usepackage{mathtools}

\hypersetup{
colorlinks=true,
linkcolor=blue,
filecolor=magenta,      
urlcolor=cyan,
}

\usepackage{caption}
\usepackage{algorithm}
\usepackage{algpseudocode}

\newtheorem{theorem}{Theorem}[section]
\newtheorem{lemma}[theorem]{Lemma}

\newtheorem{corollary}[theorem]{Corollary}

\theoremstyle{definition}
\newtheorem{definition}[theorem]{Definition}

\newtheorem{example}[theorem]{Example}
\newtheorem{conjecture}[theorem]{Conjecture}

\usepackage[left=1in, right=1in, top=1in, bottom=1in]{geometry} 
\usepackage{graphicx} 
\usepackage{tikz}

\newcommand{\Av}{\mathrm{Av}}

\newcommand*{\Cdot}{\raisebox{-0.7ex}{\scalebox{2}{$\cdot$}}}

\title{Extending results on Wilf-equivalence of partial shuffles}

\author[M. Albert]{Michael Albert}
\address{School of Computing, University of Otago, 133 Union Street East, Dunedin 9016, New Zealand}
\email{michael.albert@otago.ac.nz}
\author[D. Searles]{Dominic Searles}
\address{Department of Mathematics and Statistics, University of Otago, 730 Cumberland St., Dunedin 9016, New Zealand}
\email{dominic.searles@otago.ac.nz}
\author[M. Slattery-Holmes]{Matthew Slattery-Holmes}
\address{Department of Mathematics and Statistics, University of Otago, 730 Cumberland St., Dunedin 9016, New Zealand}
\email{slama077@student.otago.ac.nz}

\subjclass[2020]{05A05, 05A15}

\date{December 23, 2025}
\keywords{permutation patterns, enumeration, Wilf-equivalence}

\begin{document}

\maketitle
\begin{abstract}

In 2020, Bloom and Sagan defined subsets of the symmetric group $\mathfrak{S}_n$ called partial shuffles, and proved a formula for the Schur expansion of the pattern quasisymmetric function associated with a partial shuffle. In their proof, they establish that any two partial shuffles of the same size are Wilf-equivalent. We give an alternative proof of this fact, using an iterative approach. We also show that Wilf-equivalence is preserved on including a decreasing pattern in partial shuffles, and we provide some enumerative results for avoidance classes whose bases consist of a partial shuffle and a decreasing permutation.
\end{abstract}

\section{Introduction}
Let $\mathfrak{S}_n$ denote the set of permutations of size $n$. We write a permutation $\pi\in \mathfrak{S}_n$ as $\pi=\pi_1\ldots\pi_n$, where $\pi_i=\pi(i)$ is the \textit{value} of the $i^{th}$ \textit{element} of $\pi$. A permutation $\pi = \pi_1\ldots\pi_n$ is said to \textit{contain} a permutation $\rho = \rho_1\ldots\rho_k$ as a \textit{pattern} if there are indices $i_1<\cdots< i_k$ such that in the \textit{subpermutation} $\pi_{i_1}\ldots \pi_{i_k}$ of $\pi$, $\pi_{i_j} < \pi_{i_{j+1}}$ if and only if $\rho_j < \rho_{j+1}$. If $\pi$ does not contain $\rho$, then $\pi$ \textit{avoids} $\rho$, and $\pi$ avoids a set of patterns $\Pi$ if $\pi$ avoids every pattern in $\Pi$.
The set of all permutations that avoid $\Pi$ is denoted $\mathrm{Av}(\Pi)$, and the subset of all permutations of size $n$ that avoid $\Pi$ is denoted $\mathrm{Av}_n(\Pi)$. 
We write $\#S$ to denote the number of elements in the set $S$, and we say that two sets of patterns $\Pi$ and $\Pi'$ are \textit{Wilf-equivalent} if $\#\mathrm{Av}_n(\Pi) = \#\mathrm{Av}_n(\Pi')$ for all $n$. 
The interval of integers $x$ such that $a\leqslant x \leqslant b$ is written $[a,b]$, and if $a=1$ we simply write $[b]$. We let $\iota_k$ denote the increasing permutation of size $k$, and $\delta_k$ the decreasing permutation of size $k$. 

Bloom and Sagan \cite{bloom2020revisiting} defined subsets of $\mathfrak{S}_n$ called partial shuffles, generalising a definition of \cite{HPS20}. 
For $a\geqslant 1$ and $b\geqslant 0$ such that $a+b \geqslant 2$, the \textit{partial shuffle} $\Pi(a,b)$ of size $a+b$ is the set of permutations of size $a+b$ that one obtains by taking the elements $[a+b]\backslash\{a\}$ in increasing order, and placing $a$ into every position except for that which would give $\iota_{a+b}$. For example, $\Pi(3,2) = \{12453,12435,13245,31245\}$. 
In \cite{bloom2020revisiting} it was shown that the pattern quasisymmetric function (\cite{HPS20}) associated with a partial shuffle is Schur-positive, proving a more general version of a conjecture of \cite{HPS20}. 

One interesting result established in the work of \cite{bloom2020revisiting} is that any two partial shuffles of the same size are Wilf-equivalent. In Section 2 of this article we provide an alternative proof of this fact, then in Section 3 we extend it to include any decreasing permutation $\delta_k$ in the set of patterns to be avoided.
We conclude with some enumerative results, showing that for sufficiently large $n$, $\#\mathrm{Av}_n(\Pi(a,b),\delta_k)$ is given by a polynomial in $n$, of degree $(a+b-2)(k-2)$, and in the special case when $k=3$, we show that the leading coefficient of this polynomial is a Catalan number.

\section{Partial shuffles and Wilf-equivalence}

In this section, we will prove that any two partial shuffles of the same size are Wilf-equivalent. Although this result is known (\cite[Lemma 4.4]{bloom2020revisiting}), our proof proceeds differently, using an iterative approach. 
We will show that for $a\geqslant 2$ and $b\geqslant 0$, $\Pi(a,b)$ and $\Pi(a-1,b+1)$ are Wilf-equivalent, by describing injective functions between their respective avoidance classes. This implies by induction that $\Pi(a+b,0)$ is Wilf-equivalent to $\Pi(c,d)$ for all pairs $c,d$ such that $c+d\geqslant 2$, $c\geqslant 1$, $d\geqslant 0$, and $c+d$ sums to $a+b$. 

Throughout this work we make use of \textit{permutation diagrams}. Given a permutation $\pi=\pi_1\ldots\pi_n$, the diagram of $\pi$ is the set of points $\{(i,\pi_i) \, : \, 1 \leqslant i \leqslant  n\}$; see Figure \ref{fig: a and a-1}. If an increasing or decreasing pattern appears in $\pi$, we may represent this with a line of positive or negative gradient, respectively, in the diagram of $\pi$. At times, we use the cardinal directions to describe relative position of permutation elements in permutation diagrams, e.g., if $\pi_i<\pi_j$ and $i>j$, we say $\pi_i$ is southeast of $\pi_j$. Throughout, we often present figures involving permutation diagrams with some elements bolded and others in grey. The bolded elements should be treated as the primary focus, and the elements in grey considered as providing additional context. 

We will often discuss the smallest element of a permutation that acts as an $a$ in a pattern from $\Pi(a-1,b+1)$, which we denote $\underline{a}$. We also consider every element that acts as an $a-1$ in some pattern from $\Pi(a-1,b+1)$ for which $\underline{a}$ acts as the $a$. We say that these are the $a-1$ elements \textit{associated} with $\underline{a}$, and the smallest of these, we denote $\underline{a-1}$.

\begin{example}\label{Pi22ex}
Let $a=3$, $b=1$, and consider the permutation $\pi = 582916743\in \Av(\Pi(3,1))$. This permutation contains patterns from $\Pi(2,2)$. The smallest element in any of these patterns that acts as an $a$  is $\underline{3} = \pi_6=6$, indicated in blue in Figure~\ref{fig: a and a-1}, and the associated $a-1$ elements are $\pi_1,\pi_3, \pi_8,$ and $\pi_9$, shown in red. 
\begin{figure}[H]
\centering
\begin{tikzpicture}[scale=1.0]

\draw[thick] (0,0) rectangle (5,5);

\fill[red] (0.5,2.5) circle (0.1);
\fill[black] (1,4) circle (0.1);
\fill[red] (1.5,1) circle (0.1);
\fill[black] (2,4.5) circle (0.1);
\fill[black] (2.5,0.5) circle (0.1);
\fill[blue] (3,3) circle (0.1);
\fill[black] (3.5,3.5) circle (0.1);
\fill[red] (4,2) circle (0.1);
\fill[red] (4.5,1.5) circle (0.1);

\end{tikzpicture}
\caption[below]{The permutation $582916743\in \Av_9(\Pi(3,1))$. Of the patterns from $\Pi(2,2)$, the smallest $a$ element is shown in blue, where $a = 3$, and all associated $a-1$ elements are shown in red.}\label{fig: a and a-1}
\end{figure}
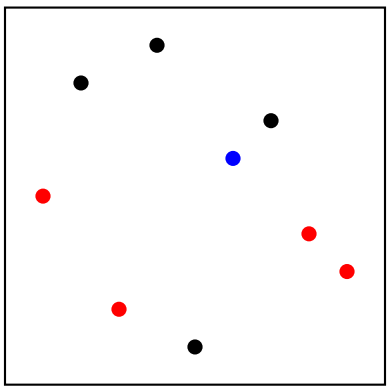
\end{example}

\begin{lemma}\label{intervallemma}
Let $\pi\in \mathfrak{S}_n$ such that $\pi\notin \Av(\Pi(a-1,b+1)$. Then the $a-1$ elements associated with $\underline{a}$ form an interval. Specifically, the elements of $[\underline{a-1},~\underline{a}-1]$ are the $a-1$ elements associated with $\underline{a}$. 
\end{lemma}

\begin{proof}
Consider a pattern in $\pi$ from $\Pi(a-1,b+1)$ for which $\underline{a}$ acts as the $a$ and $\underline{a-1}$ acts as the $a-1$. 
In the permutation diagram of $\pi$, such a pattern presents as follows. If $a>2$, there is an increasing subpermutation $\iota_{a-2}$, an increasing subpermutation $\iota_{b+1}$ (whose lowest element is $\underline{a}$) northeast of the $\iota_{a-2}$, along with $\underline{a-1}$ which is above $\iota_{a-2}$, below $\iota_{b+1}$, and either left of the largest element of $\iota_{a-2}$ or right of $\underline{a}$. If $a=2$, there is an increasing subpermutation $\iota_{b+1}$ whose lowest element is $\underline{a}$, along with $\underline{a-1}$ which is below $\iota_{b+1}$ and right of $\underline{a}$. 

Let $r$ denote the rightmost element of $\pi$ that acts as an $a-2$ for such a pattern from $\Pi(a-1,b+1)$ (in the case $a=2$, $(r, \pi(r))$ can be considered to be the origin $(0,0)$). See Figure~\ref{fig: a-1 interval}. Let $q\in [\underline{a-1}+1,\underline{a}-1]$. If $q$ is left of $r$ or right of $\underline{a}$, then by definition it is an $a-1$ element associated with $\underline{a}$. If $q$ is right of $r$ and left of $\underline{a}$, then since $q>\underline{a-1}$, an instance of the same pattern from $\Pi(a-1,b+1)$ is found by replacing $\underline{a}$ with $q$, but this means $q$ plays the role of an $a$ in a pattern from $\Pi(a-1,b+1)$, contradicting the minimality of $\underline{a}$.
\end{proof}
\begin{figure}[h]
\centering
\begin{tikzpicture}[scale=0.8]

\draw[thick] (0,0) rectangle (7,7);
\draw[thick,dashed](7,3)--(0,3);
\draw[thick,dashed] (7,5) -- (0,5);

\fill[gray,opacity=0.25] (3.5,2.5) -- (3.5,5) -- (5,5) -- (5,2.5) -- cycle;
\fill[white] (5,5.1) circle (0.15) node[black, above left] {$\underline{a}$};
\fill[red] (5,5) circle (0.15);
\fill[white] (2,2.9) circle (0.15) node[black, below] {$\underline{a-1}$};
\fill[red] (2,3) circle (0.15) ;
\fill[white] (3.5,2.4) circle (0.15) node[black, below] {$r$};
\fill[red] (3.5,2.5) circle (0.15);

\fill[white] (1.44,4) circle (0.15) node[black,left]{$q$};
\fill[red] (1.5,4) circle (0.15);

\draw[ultra thick,black] (1.5,0.5) -- (3.39,2.39) node[midway, below right] {$\iota_{a-2}$};

\draw[ultra thick,black] (5.3,5.3) -- (6.6,6.6)  node[midway, above left] {$\iota_{b}$};

\end{tikzpicture}
\caption[below]{If $\underline{a}$ is the smallest $a$ from any pattern in $\Pi(a-1,b+1)$ appearing in a permutation $\pi$, and $r$ plays the role of the rightmost $a-2$, then nothing may occur in the shaded region, and any element $q$ such that $r<q<\underline{a}$ is an $a-1$ element associated with $\underline{a}$.}\label{fig: a-1 interval}
\end{figure}
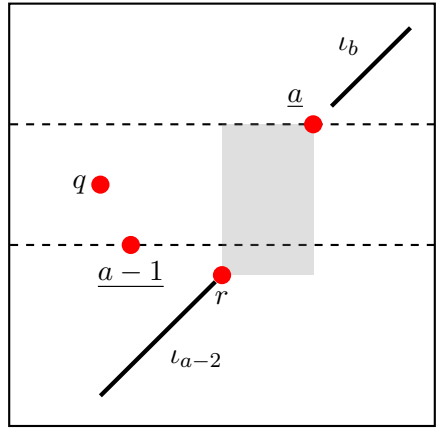

Given $a,b$ such that $a+b\geqslant 2$, $a\geqslant 2$ and $b\geqslant 0$, we now use Lemma~\ref{intervallemma} to define a function $S$ that acts on permutations. Note that the function $S$ depends on the values of $a$ and $b$, but we will simply use $S$ throughout, in order to reduce excessive notation. 

\begin{definition}\label{Smap}
Let $\pi\in \mathfrak{S}_n$. If $\pi\in \Av(\Pi(a-1,b+1))$, then $S(\pi)=\pi$. If $\pi\notin \Av(\Pi(a-1,b+1))$, then informally, $S(\pi)$ is obtained by rotating the values in the interval $[\underline{a-1},\underline{a}]$, i.e. by increasing the value of every $a-1$ element associated with $\underline{a}$ by 1, and then replacing $\underline{a}$ with $\underline{a-1}$. 
Formally:
\[
S(\pi)_i =
\begin{cases}
    \pi_i + 1 & \text{if $\pi_i \in [\underline{a-1},\underline{a}-1]$} \\
    \underline{a-1}  & \text{if }\pi_i = \underline{a} \\
    \pi_i & \text{otherwise}
\end{cases}
\]
\end{definition}

By Lemma \ref{intervallemma}, $S$ only changes the values of permutation elements whose values lie within an interval, so $S$ sends permutations to permutations. If $\pi_i=q$ then we will slightly abuse notation by referring to the image of $q$ under the map $S$ as $S(q)$ rather than the more correct $S(\pi)_i$. Note that the map $S$ removes patterns in $\Pi(a-1,b+1)$ from $\pi$, replacing them with patterns from $\Pi(a,b)$. 

\begin{example}
Continuing with Example~\ref{Pi22ex}, if $a=3$ and $b=1$ then the map $S$ acts on the permutation $\pi=582916743\in \Av(\Pi(3,1))$ by rotating elements in the interval $[\underline{2},\underline{3}]$ to give $S(\pi) =683912754$ (see Figure~\ref{fig: pi and S(pi)}). Note that $S(\pi)$ contains patterns from $\Pi(3,1)$ and $\Pi(2,2)$.

\begin{figure}[h]
\centering
\begin{tikzpicture}[scale=1.0]

\draw[thick] (0,0) rectangle (5,5);
\draw[thick,dashed](5,3)--(0,3);
\draw[thick,dashed] (5,1) -- (0,1);

\draw[thick] (7,0) rectangle (12,5);
\draw[thick,dashed](12,3)--(7,3);
\draw[thick,dashed] (12,1) -- (7,1);

\fill[white] (-0.6,2.24) node[black]{$\pi = $};

\fill[white] (6.2,2.24) node[black]{$S(\pi)=$};

\fill[red] (0.5,2.5) circle (0.1);
\fill[black] (1,4) circle (0.1);
\fill[red] (1.5,1) circle (0.1);
\fill[black] (2,4.5) circle (0.1);
\fill[black] (2.5,0.5) circle (0.1);
\fill[blue] (3,3) circle (0.1);
\fill[black] (3.5,3.5) circle (0.1);
\fill[red] (4,2) circle (0.1);
\fill[red] (4.5,1.5) circle (0.1);

\fill[red] (7.5,3) circle (0.1);
\fill[black] (8,4) circle (0.1);
\fill[red] (8.5,1.5) circle (0.1);
\fill[black] (9,4.5) circle (0.1);
\fill[black] (9.5,0.5) circle (0.1);
\fill[blue] (10,1) circle (0.1);
\fill[black] (10.5,3.5) circle (0.1);
\fill[red] (11,2.5) circle (0.1);
\fill[red] (11.5,2) circle (0.1);

\end{tikzpicture}
\caption{$\pi = 582916743 \in \Av(\Pi(3,1))$ and its image $S(\pi)=683912754$.} \label{fig: pi and S(pi)}

\end{figure}
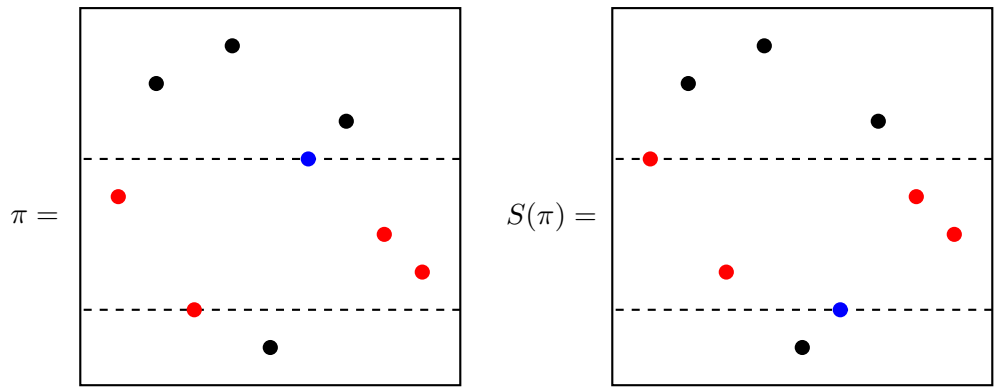
\end{example}

Two sequences of the same length are said to be \emph{order-isomorphic} if their entries are in the same relative order. 
The following fact is an immediate consequence of Lemma \ref{intervallemma}.

\begin{corollary}\label{ignorej}
Let $\pi\in \mathfrak{S}_n$ and let $\pi_{i_1}\pi_{i_2}\ldots \pi_{i_m}$ be a subpermutation of $\pi$. If $\pi_{i_1}\pi_{i_2}\ldots \pi_{i_m}$ is not order-isomorphic to $S(\pi_{i_1})S(\pi_{i_2})\ldots S(\pi_{i_m})$, then the subpermutation $\pi_{i_1}\pi_{i_2}\ldots \pi_{i_m}$ contains $\underline{a}$ and at least one associated $a-1$. 
\end{corollary}

\begin{definition}
For $a\geqslant 2$, the unique permutation in the intersection of the sets $\Pi(a,b)$ and $\Pi(a-1,b+1)$ is the permutation obtained by transposing the elements $a$ and $a-1$ in the increasing permutation $\iota_{a+b}$. We call this permutation $\sigma_{a,b}$. 
\end{definition}
For example, $\sigma_{3,1} = 1324$, which is the unique permutation in $\Pi(3,1)\cap\Pi(2,2)$.  

\begin{lemma}\label{nosigmaab}
If $\pi\in \Av(\sigma_{a,b})$ then $S(\pi)\in \Av(\sigma_{a,b})$. 
\end{lemma}
\begin{proof}
Let $\pi\in \Av(\sigma_{a,b})$. If $\pi\in \Av(\Pi(a-1,b+1))$ the lemma is trivially true, since $S(\pi)=\pi$. Suppose $\pi\notin \Av(\Pi(a-1,b+1))$, and $S(\pi)$ contains $\sigma_{a,b}$. We know by Corollary \ref{ignorej} that the pattern $\sigma_{a,b}$ in $S(\pi)$ must involve $S(\underline{a})$. Recall that patterns from $\Pi(a-1,b+1)$ have the form $12\ldots(a-2)a\ldots (a+b)$ with $a-1$ placed somewhere other than between $a-2$ and $a$. Therefore there is an increasing subpermutation $\iota_{b}$ northeast of $\underline{a}$, and an increasing subpermutation $\iota_{a-2}$ southwest of $\underline{a}$. In $S(\pi)$ the $a$ and $a-1$ elements of $\sigma_{a,b}$ have an $\iota_{b}$ to the northeast, and an $\iota_{a-2}$ to the southwest. 

Suppose in $S(\pi)$, that $S(\underline{a})$ is part of one of the increasing subpermutations in $\sigma_{a,b}$, either $\iota_{a-2}$ southwest of the $a(a-1)$, or $\iota_{b}$ northeast of the $a(a-1)$. There is an $\iota_b$ northeast of $\underline{a}$ in $\pi$, and therefore of $S(\underline{a})$ in $S(\pi)$, and similarly, an $\iota_{a-2}$ below $\underline{a-1}$ and left of $\underline{a}$ in $\pi$, and therefore southwest of $S(\underline{a})$ in $S(\pi)$. It follows that $\pi$ contains $\sigma_{a,b}$ as a pattern, a contradiction. 

Next, suppose that $S(\underline{a})$ is the $a-1$ element in $\sigma_{a,b}$. This means that the $a$  element of $\sigma_{a,b}$ is northwest of $S(\underline{a})$ in $S(\pi)$. We will denote this $a$ element by $a_\sigma$. If the pre-image under $S$ of $a_\sigma$ was not an $a-1$ associated with $\underline{a}$ in $\pi$, then it must be northwest of $\underline{a}$. Since the increasing subpermutation $\iota_{b}$ is above $a_\sigma$ and right of $S(\underline{a})$, it must be northeast of $\underline{a}$. Similarly, the increasing subpermutation $\iota_{a-2}$ is below $S(\underline{a})$ and left of $a_\sigma$, so in $\pi$ it is southwest of both $a_\sigma$ and $\underline{a}$. This gives an occurrence of $\sigma_{a,b}$ in $\pi$ (using $\underline{a}$ instead of $S(\underline{a})$ and all the other elements of $\sigma_{a,b}$ the same as in $S(\pi)$), which is a contradiction. 

If $a_\sigma$ was an $a-1$ for $\underline{a}$ in $\pi$, we know there is some $a-2$ element $t$ left of $\underline{a}$ and right of $a_\sigma$, such that $t$ is the largest element of an increasing subpermutation $\iota_{a-2}$ in $\pi$. There is another copy of $\iota_{a-2}$ southwest of both $a_\sigma$ and $S(\underline{a})$, whose largest element $m$ plays the role of $a-2$ in $\sigma_{a,b}$ in $S(\pi)$. 

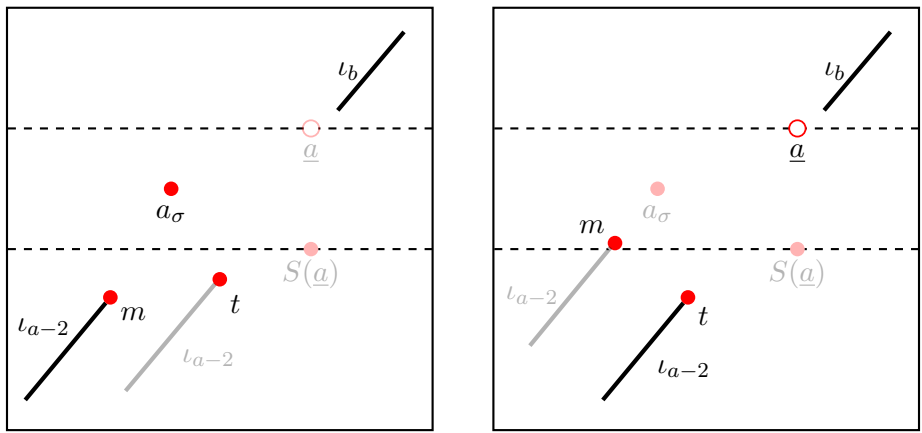
\begin{figure}[h]
\centering
\begin{tikzpicture}[scale=0.8]

\draw[thick] (0,0) rectangle (7,7);
\draw[thick,dashed](7,3)--(0,3);
\draw[thick,dashed] (7,5) -- (0,5);

\fill[white] (5,4.9) circle (0.12) node[black,below,opacity=0.3] {$\underline{a}$};
\fill[red,opacity=0.3] (5,5) circle (0.15);
\fill[white] (5,5) circle (0.12) ;
\fill[white] (5,3) circle (0.12) node[black,below,opacity=0.3] {$S(\underline{a})$};
\fill[red,opacity=0.3] (5,3) circle (0.12) ;
\fill[white] (2.7,3.9) circle (0.12) node[black,below] {$a_\sigma$};
\fill[red] (2.7,4) circle (0.12) ;
\fill[white] (3.5,2.4) circle (0.12) node[black,below right] {$t$};
\fill[red] (3.5,2.5) circle (0.12) ;

\begin{scope}[xshift = -8cm]
\draw[ultra thick,black,opacity=0.3] (9.95,0.65) -- (11.42,2.42) node[midway, below right] {$\iota_{a-2}$};
\draw[ultra thick,black] (8.3,0.5) -- (9.7,2.2) ;
\fill[white] (8.6,1.7) circle (0.01) node[black] {$\iota_{a-2}$};

\draw[ultra thick,black] (13.44,5.3) -- (14.53,6.6)  node[midway,left] {$\iota_{b}$};
\end{scope}
\fill[white] (1.7,2.2) circle (0.12) node[black,below right]{$m$};
\fill[red] (1.7,2.2) circle (0.12);

\draw[thick] (8,0) rectangle (15,7);
\draw[thick,dashed](15,3)--(8,3);
\draw[thick,dashed] (15,5) -- (8,5);

\fill[white] (13,4.9) circle (0.12) node[black,below] {$\underline{a}$};
\fill[red] (13,5) circle (0.15);
\fill[white] (13,5) circle (0.12) ;
\fill[white] (13,3) circle (0.12) node[black,below,opacity=0.3] {$S(\underline{a})$};
\fill[red,opacity=0.3] (13,3) circle (0.12) ;
\fill[white] (10.7,3.9) circle (0.12) node[black,below,opacity=0.3] {$a_\sigma$};
\fill[red,opacity=0.3] (10.7,4) circle (0.12) ;

\fill[white] (10,3.1) circle (0.15) node[black,above left]{$m$};
\fill[red] (10,3.1) circle (0.12);

\draw[ultra thick,black] (9.8,0.5) -- (11.2,2.2) node[midway, below right] {$\iota_{a-2}$};
\draw[ultra thick,black,opacity=0.3] (8.6,1.4) -- (9.93,3.03) node[midway, left] {$\iota_{a-2}$};

\draw[ultra thick,black] (13.44,5.3) -- (14.53,6.6)  node[midway,left] {$\iota_{b}$};

\fill[white] (11.2,2.2) circle (0.12) node[black,below right] {$t$};
\fill[red] (11.2,2.2) circle (0.12) ;

\end{tikzpicture}
\caption[below]{Left: If $S(\underline{a})$ is the $a-1$ element of $\sigma_{a,b}$, and $m<t$, then $\pi$ contains $\sigma_{a,b}$. Right: If $m>t$ in $\pi$, then $m$ cannot be below $S(\underline{a})$ in $S(\pi)$. }\label{fig: S(underline a) is a-1)}
\end{figure}
If $m<t$, (see Figure \ref{fig: S(underline a) is a-1)}, left) this would mean that $\pi$ contains $\sigma_{a,b}$ with $a_\sigma$ acting as the $a$, and $t$ acting as the $a-1$, contradicting our assumption that $\pi\notin\Av(\sigma_{a,b})$. We also cannot have $m>t$ (see Figure \ref{fig: S(underline a) is a-1)}, right) because then $m$ would have been an $a-1$ in $\pi$ with $\underline{a}$ acting as the $a$ element, and the $\iota_{a-2}$ ending with $t$ and the $\iota_{b}$ above $\underline{a}$ providing the other sections of the pattern. If $m$ did act as an $a-1$, then its image under $S$ would be above $S(\underline{a})$ in $S(\pi)$, contradicting the fact that it acts as the $a-2$ in $\sigma_{a,b}$. Hence, $S(\underline{a})$ cannot be the $a-1$ element of $\sigma_{a,b}$.  

Suppose $S(\underline{a})$ acts as $a$ in $\sigma_{a,b}$. There is an element $(a-1)_{\sigma}$ that acts as an $a-1$ southeast of $S(\underline{a})$ in $\sigma_{a,b}$ in $S(\pi)$. There is also an $\iota_{a-2}$ southwest of both of these elements. Any occurrence of $\sigma_{a,b}$ in $S(\pi)$, by Corollary \ref{ignorej}, must contain $S(p)$ for some element $p$ that acted as an $a-1$ associated with $\underline{a}$ in $\pi$. This has to be to the right of $S(\underline{a})$ and $(a-1)_\sigma$, since $S(p) > S(\underline{a})$, meaning that $S(p)$ is part of the increasing subpermutation $\iota_{b}$ northeast of the $a$ element, $S(\underline{a})$. This pattern $\iota_{b}$ doesn't involve $S(\underline{a})$, meaning $p$ is also part of the pattern $\iota_{b}$ in $\pi$. 

As $\underline{a}$ was an $a$ element from a pattern in $\Pi(a-1,b+1)\backslash\{\sigma_{a,b}\}$ in $\pi$, for $p$ to have been an $a-1$, there must be an increasing subpermutation $\iota_{b}$ whose smallest element $m$ is northwest of $\underline{a}$ and northeast of $p$.

\begin{figure}[h]
\centering
\begin{tikzpicture}[scale=0.8]

\begin{scope}[xshift=-8cm]
   \draw[thick] (8,0) rectangle (15,7);
\draw[thick,dashed](15,3)--(8,3);
\draw[thick,dashed] (15,5) -- (8,5);

\fill[white] (10.6,4.9) circle (0.12) node[black,below] {$\underline{a}$};
\fill[red] (10.6,5) circle (0.15);
\fill[white] (10.6,5) circle (0.12) ;
\fill[white] (10.6,3) circle (0.12) node[black,below,opacity=0.3] {$S(\underline{a})$};
\fill[red,opacity=0.3] (10.6,3) circle (0.12) ;
\fill[white] (11.5,2.6) circle (0.12) node[black,below right] {$(a-1)_\sigma$};
\fill[red] (11.6,2.6) circle (0.12) ;
\fill[white] (13.2,3.8) circle(0.12) node[below,black,opacity=0.3]{$p$};
\fill[red,opacity=0.3] (13.2,3.9) circle(0.12);

\draw[ultra thick,black] (8.3,0.5) -- (9.74,2) node[midway, below right] {$\iota_{a-2}$};

\draw[ultra thick,black] (12.3,5.3) -- (13.6,6.6)  node[near end, below right] {$\iota_{b}$};
\fill[red] (12.3,5.3) circle (0.12) node[black,above left]{$m$};

\draw[ultra thick,black,opacity=0.3] (13.27,3.98) -- (14.8,5.5)  node[near start, below right] {$\iota_{b}$};
\end{scope}


\draw[thick] (8,0) rectangle (15,7);
\draw[thick,dashed](15,3)--(8,3);
\draw[thick,dashed] (15,5) -- (8,5);

\fill[white] (10.6,4.9) circle (0.12) node[black,below] {$\underline{a}$};
\fill[red] (10.6,5) circle (0.15);
\fill[white] (10.6,5) circle (0.12) ;
\fill[white] (10.6,3) circle (0.12) node[black,below,opacity=0.3] {$S(\underline{a})$};
\fill[red,opacity=0.3] (10.6,3) circle (0.12) ;
\fill[white] (12.4,2.6) circle (0.12) node[black,below right] {$(a-1)_\sigma$};
\fill[red] (12.5,2.6) circle (0.12) ;
\fill[white] (13.2,3.8) circle(0.12) node[below,black,opacity=0.3]{$p$};
\fill[red,opacity=0.3] (13.2,3.9) circle(0.12);

\draw[ultra thick,black] (8.3,0.5) -- (9.74,2) node[midway, below right] {$\iota_{a-2}$};

\draw[ultra thick,black] (11.7,5.3) -- (13,6.6)  node[near end, below right] {$\iota_{b}$};
\fill[red] (11.7,5.3) circle (0.12) node[black,above left]{$m$};

\draw[ultra thick,black,opacity=0.3] (13.27,3.98) -- (14.8,5.5)  node[near start, below right] {$\iota_{b}$};

\end{tikzpicture}
\caption[below]{Left: Suppose $S(\underline{a})$ is the $a$ element of $\sigma_{a,b}$. If $m$ is to the right of $(a-1)_\sigma$, then $\pi$ contains $\sigma_{a,b}$. Right: If $m$ is to the left of $(a-1)_\sigma$ then $(a-1)_\sigma$ is an $a-1$ associated with $\underline{a}$ in $\pi$.}\label{qrightr}
\end{figure}
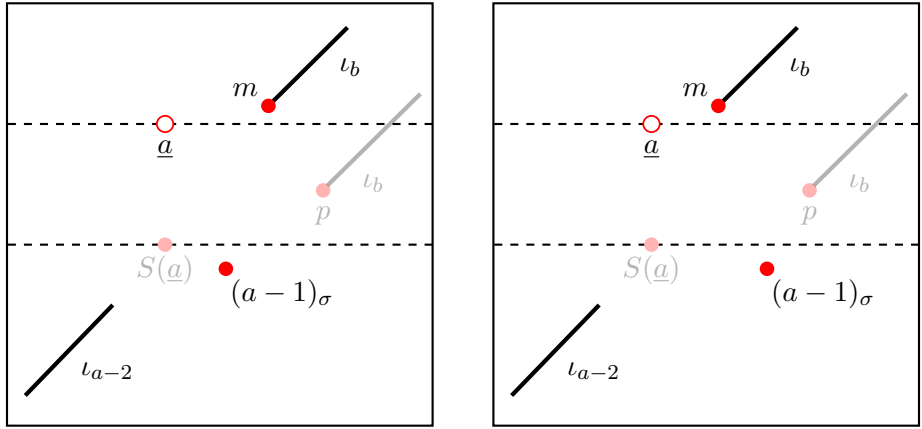

If $(a-1)_\sigma$ is to the left of $m$, then $\underline{a}$ and $(a-1)_\sigma$ would be the $a$ and $a-1$ respectively, in an occurrence of $\sigma_{a,b}$ in $\pi$ (see Figure \ref{qrightr}, left), a contradiction, since $\pi\in \Av(\sigma_{a,b})$. If $(a-1)_\sigma$ is to the right of $m$, then $(a-1)_\sigma$ is an $a-1$ associated with $\underline{a}$ in $\pi$, which contradicts the fact that $S(\underline{a})$ is above $(a-1)_\sigma$ in $S(\pi)$ (see Figure~\ref{qrightr}, right).  
 
Since $\pi\in \Av(\sigma_{a,b})$, by Corollary \ref{ignorej} we know that if $\sigma_{a,b}$ appeared in $S(\pi)$, it would contain $S(\underline{a})$. We have shown that any role $S(\underline{a})$ could play in $\sigma_{a,b}$ leads to a contradiction, so $S(\pi)\in \Av(\sigma_{a,b})$.  
\end{proof}

\begin{lemma}\label{terminates}
If $\pi\in \Av(\sigma_{a,b})$ then $S^{n-a}(\pi)\in \Av(\Pi(a-1,b+1))$.
\end{lemma}
\begin{proof}
If $\pi\in \Av(\Pi(a-1,b+1))$ or if $S(\pi)\in \Av(\Pi(a-1,b+1))$ the claim is trivially satisfied, so suppose that both $\pi$ and $S(\pi)$ contain a pattern from $\Pi(a-1,b+1)$. We will show that the $\underline{a}$ element in $\pi$ is less than the $\underline{a}$ element in $S(\pi)$ (denoted $\underline{a}'$), by showing that the preimage of $\underline{a}'$ is above $\underline{a}$ in $\pi$. Noting that by Lemma~\ref{nosigmaab}, $S^k(\pi)\in \Av(\sigma_{a,b})$ for all $k$, this implies that we reach a fixed point (i.e., an element of $\Av(\Pi(a-1,b+1))$) after at most $n-a$ iterations of $S$. Let $\rho$ denote a pattern from $\Pi(a-1,b+1)$ in $S(\pi)$ in which $\underline{a}'$ acts as the $a$. Note that showing that the preimage of $\underline{a}'$ is above $\underline{a}$ in $\pi$ amounts to showing that if $q\leqslant  \underline{a}$ in $\pi$ then $S(q)$ may not act as the $a$ element of a pattern from $\Pi(a-1,b+1)$ in $S(\pi)$. 

First, consider elements $q<\underline{a-1}$ in $\pi$. If the image $S(q)$ of such an element was the $a$ in some pattern $\rho\in \Pi(a-1,b+1)$ in $S(\pi)$, then since $\rho$ must contain $S(\underline{a})$ by Corollary \ref{ignorej}, and since $S(q)<S(\underline{a})$, $S(\underline{a})$ must be part of the increasing pattern $\iota_b$, northeast of $S(q)$. However, there was already an $\iota_b$ northeast of $\underline{a}$ in $\pi$, meaning $q$ would be the $a$ element in some pattern from $\Pi(a-1,b+1)$ in $\pi$, a contradiction since $q<\underline{a}$ by assumption. 

Now we will show that the element $S(\underline{a})$ cannot act as $\underline{a}'$ in a pattern $\rho\in \Pi(a-1,b+1)$ in $S(\pi)$. Suppose it did. Then there must be an increasing subpermutation $\iota_{a-2}$ and an $a-1$ element below $S(\underline{a})$. Everything below $S(\underline{a})$ is in the same order relative to $S(\underline{a})$ as it was to $\underline{a}$ in $\pi$. The part of $\rho$ above $S(\underline{a})$ is an increasing subpermutation $\iota_{b}$ with least element $r$, such that the $a-1$ is not between $S(\underline{a})$ and $r$. Now, either this increasing subpermutation $\iota_{b}$ was above $\underline{a}$ in $\pi$, in which case the preimage of the pattern $\rho$ is also $\rho$, or the pre-image of $r$ was an $a-1$ for $\underline{a}$ and there is some other $\iota_{b}$ above $\underline{a}$ in $\pi$, with least element right of $\underline{a}$ and left of $r$. In either case, the $a-1$ element from $\rho$ in $S(\pi)$ would be an $a-1$ for $\underline{a}$ in $\pi$, and its image would be above $S(\underline{a})$ in $S(\pi)$, a contradiction. 

Lastly, consider the case when $q$ is an $a-1$ associated with $\underline{a}$ in $\pi$. We will show that $S(q)\neq \underline{a}'$. If the image of $q$ under $S$ was $\underline{a}'$ in some pattern $\rho$ from $\Pi(a-1,b+1)$, then by Corollary \ref{ignorej}, $\rho$ must involve $S(\underline{a})$.  

Suppose $q$ is left of $\underline{a}$ in $\pi$. Since $S(\underline{a})$ is below $S(q)$ in $S(\pi)$, $S(\underline{a})$ must be the $a-1$ element in $\rho$, as this is the only decreasing pair in any $\rho\in \Pi(a-1,b+1)$. We know there is an increasing subpermutation $\iota_{b}$ northeast of $\underline{a}$ and $q$ in $\pi$, due to their roles in some $\Pi(a-1,b+1)$ pattern in $\pi$. If $S(q)$ is the $a$ element in $\rho$ there is an increasing subpermutation $\iota_{a-2}$ southwest of $S(q)$ in $S(\pi)$, and therefore southwest of $q$ in $\pi$. This $\iota_{a-2}$ is also below $S(\underline{a})$, the $a-1$ element of $\rho$. This implies $S(\pi)$ contains $\sigma_{a,b}$. By Lemma \ref{nosigmaab} we would then have $\pi\notin \Av(\sigma_{a,b})$, a contradiction. 

Suppose $q$ is to the right of $\underline{a}$ and $S(q) = \underline{a}'$. Again, the pattern $\rho$ in which $S(q)$ is $\underline{a}'$ must contain $S(\underline{a})$. As $S(\underline{a})< S(q)$, $S(\underline{a})$ must either be an $a-1$ element in $\rho$, or part of the increasing pattern $\iota_{a-2}$. Since there is already an $\iota_{a-2}$ with maximal element $h$ southwest of $S(\underline{a})$, if $S(\underline{a})$ was part of the $\iota_{a-2}$, we would have the same $\rho$ with $a$ element $q$ in $\pi$. Therefore we suppose that $S(\underline{a})$ plays the role of an $a-1$ for the pattern $\rho$ in $S(\pi)$.

If $S(q) = \underline{a}'$ and $S(\underline{a})$ is an $a-1$ associated with $\underline{a}'$ in $\rho\in \Pi(a-1,b+1)$, then since $S(\underline{a})$ is left of $S(q)$, there is an increasing subpermutation below $S(\underline{a})$ and $S(q)$ whose maximal element $g$ is right of $\underline{a}$ and left of $q$. Now then, if $h<g$ (see Figure \ref{fig: a-1 in pi}, left), $S(\pi)$ contains $\sigma_{a,b}$ with $S(\underline{a})$ acting as $a$ and $g$ as $a-1$. If $h>g$ (see Figure \ref{fig: a-1 in pi}, right), then $q$ is a lower $a$ than $\underline{a}$ in $\pi$, with $h$ playing the role of $a-1$. Both of these situations contradict our assumptions on $\pi$.  
\end{proof}

\begin{figure}[h]
\centering
\begin{tikzpicture}[scale=0.9]
\draw[thick] (0,0) rectangle (6,6);
\draw[thick,dashed] (6,3)--(0,3);
\draw[thick,dashed] (6,4) -- (0,4);

\fill[white] (2.7,3.9) circle (0.12) node[black,below,opacity=0.3] {$\underline{a}$};
\fill[red,opacity=0.3] (2.7,4) circle (0.15);
\fill[white] (2.7,4) circle (0.12);
\fill[white] (2.7,3) circle (0.12) node[black,below] {$S(\underline{a})$};
\fill[red] (2.7,3) circle (0.12);

\fill[white] (4.05,3.4) circle (0.12) node[black,right,opacity=0.3] {$q$};
\fill[red,opacity=0.3] (4,3.5) circle (0.12);

\draw[ultra thick,black] (0.4,0.7) -- (1.73,2);
\fill[white] (0.7,1.7) node[black] {$\iota_{a-2}$};

\fill[white] (1.75,2) circle (0.12) node[below right, black] {$h$};
\fill[red] (1.73,2) circle (0.12);

\draw[ultra thick,black,opacity=0.3] (1.9,0.9) -- (3.2,2.2) node[midway, below right] {$\iota_{a-2}$};

\fill[white] (3.25,2.2) circle (0.12) node[below right, black] {$g$};
\fill[red] (3.2,2.2) circle (0.12);

\draw[ultra thick,black,opacity=0.3] (3.5,4.4) -- (4.9,5.7) node[midway,opacity=0.3, above left] {$\iota_{b}$};

\draw[ultra thick,black] (4.5,4.4) -- (5.9,5.7) node[near start, below right] {$\iota_{b}$};

\draw[thick] (7,0) rectangle (13,6);
\draw[thick,dashed](13,3)--(7,3);
\draw[thick,dashed] (13,4) -- (7,4);

\fill[white] (9.7,3.9) circle (0.12) node[black,below,opacity=0.3] {$\underline{a}$};
\fill[red,opacity=0.3] (9.7,4) circle (0.15);
\fill[white] (9.7,4) circle (0.12);
\fill[white] (9.7,3) circle (0.12) node[black,below,opacity=0.3] {$S(\underline{a})$};
\fill[red,opacity=0.3] (9.7,3) circle (0.12);
\fill[white] (11.05,3.4) circle (0.12) node[black,right] {$q$};
\fill[red] (11,3.5) circle (0.12);

\draw[ultra thick,black,opacity=0.3] (7.5,0.8) -- (8.93,2.2);
\fill[white] (7.8,1.8) node[black] {$\iota_{a-2}$};

\fill[white] (8.95,2.2) circle (0.12) node[below right, black] {$h$};
\fill[red] (8.93,2.2) circle (0.12);

\draw[ultra thick,black] (8.7,0.4) -- (10.2,1.9) node[midway, below right] {$\iota_{a-2}$};

\fill[white] (10.25,1.9) circle (0.12) node[below right, black] {$g$};
\fill[red] (10.2,1.9) circle (0.12);

\draw[ultra thick,black,opacity=0.3] (10.5,4.4) -- (11.9,5.7) node[midway, above left] {$\iota_{b}$};

\draw[ultra thick,black] (11.5,4.4) -- (12.9,5.7) node[near start, below right] {$\iota_{b}$};

\end{tikzpicture}
\caption[below]{Left: If $q$ acts as the $a$ in a pattern from $\Pi(a-1,b+1)$ in $S(\pi)$, and $h<g$, then $S(\pi)$ contains $\sigma_{a,b}$.
Right: If $h>g$, then $q$ is a lower $a$ than $\underline{a}$ in $\pi$.}\label{fig: a-1 in pi}
\end{figure}
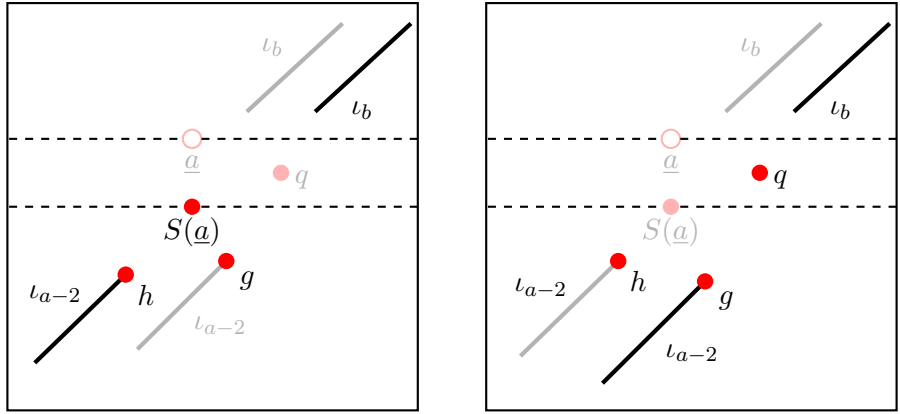

The following corollary is immediate from Lemma \ref{terminates}:
\begin{corollary}\label{cor: image of S n-a}
If $\pi\in \Av(\Pi(a,b))$, then $S^{n-a}(\pi)\in \Av(\Pi(a-1,b+1))$.
\end{corollary}
Note that by definition the map $S$ introduces patterns from $\Pi(a,b)$ as it removes patterns from $\Pi(a-1,b+1)$. The next lemma shows that although $S$ introduces patterns from $\Pi(a,b)$, they are introduced in a specific manner. 

\begin{lemma}\label{Lownewj}
If an element $t$ in $\pi$ is not the $a$ element in any pattern from $\Pi(a,b)$, but $S(t)$ is the $a$ element in some pattern from $\Pi(a,b)$ in $S(\pi)$, then $t$ is either $\underline{a}$ or an $a-1$ associated with $\underline{a}$ in $\pi$. 
\end{lemma}

\begin{proof}
Note the assumptions of the lemma imply that $\pi\neq S(\pi)$, and therefore there must be some pattern from $\Pi(a-1,b+1)$ in $\pi$. 
Suppose the claim is false, i.e., suppose $t$ does not act as an $a$ in some pattern from $\Pi(a,b)$, $S(t)$ does, and $t\notin[\underline{a-1},\underline{a}]$. 
Let $\rho$ denote a pattern from $\Pi(a,b)$ in $S(\pi)$ for which $S(t)$ is the $a$ element. Since $t$ does not act as the $a$ in an occurrence of the pattern $\rho$ in $\pi$, by Corollary \ref{ignorej}, $\rho$ involves the images under $S$ of $\underline{a}$, and some associated $a-1$ element, $q$. 

Recall that patterns from $\Pi(a,b)$ consist of an $a$ element, an increasing subpermutation $\iota_{a-1}$ below the $a$, and an increasing subpermutation $\iota_{b}$ above the $a$, such that the largest element of the $\iota_{a-1}$ is left of the smallest element of the $\iota_{b}$, and both are on the same side of $a$. 
If $t>\underline{a}$, (see Figure \ref{fig9}, left) then $S(\underline{a})$ and $S(q)$ are part of the increasing subpermutation $\iota_{a-1}$ below $S(t)$. Since $S(q)>S(\underline{a})$, this means that $q$ is right of $\underline{a}$. Since $q$ was an $a-1$ in a pattern from $\Pi(a-1,b+1)$ in $\pi$, and it was right of $\underline{a}$, it follows that $q$ was northeast of an increasing subpermutation $\iota_{a-2}$. Therefore, $q$ was part of an increasing subpermutation $\iota_{a-1}$ in $\pi$. This means we can find an occurrence of the pattern $\rho$ in $\pi$ by replacing the $\iota_{a-1}$ containing $S(\underline{a})$ and $S(q)$ with this one containing $q$, and keeping all other elements the same. This contradicts our assumption that $t$ is not the $a$ in a pattern from $\Pi(a,b)$ in $\pi$.

\begin{figure}[h]
\centering
\begin{tikzpicture}[scale=0.9]

\draw[thick] (0,0) rectangle (6,6);
\draw[thick,dashed](6,2.5)--(0,2.5);
\draw[thick,dashed] (6,3.6) -- (0,3.6);

\fill[white] (5.3,4) circle (0.13) node[black, right] {$t$};
\fill[red] (5.2,4) circle (0.12);

\draw[ultra thick,black] (4,4.4) -- (5.1,5.5) node[midway, above left] {$\iota_{b}$};

\fill[white] (2.5,3.7) node[black, above left] {$\underline{a}$};
\fill[red] (2.5,3.6) circle (0.15);

\fill[white] (3.5,2) circle (0.05) node[black,left,opacity=0.3] {$S(\underline{a})$};

\fill[white] (2.5,3.6) circle (0.12);

\fill[white] (3.1,3) circle (0.13) node[black, right,opacity=0.3] {$q$};

\draw[ultra thick,black] (0.9,0.2) -- (2.2,1.5) node[midway, below right] {$\iota_{a-2}$};

\draw[ultra thick,black,opacity=0.3] (1.5,1.5) -- (3.2,3.2);
\fill[white] (1.3,2) node[black,opacity=0.3] {$\iota_{a-1}$};
\fill[white] (3,3) circle (0.12);
\fill[white] (2.5,2.5) circle (0.12);

\fill[red,opacity=0.3] (3,3) circle (0.12);
\fill[red,opacity=0.3] (2.5,2.5) circle (0.12);

\draw[thick] (7,0) rectangle (13,6);
\draw[thick,dashed](13,2.5)--(7,2.5);
\draw[thick,dashed] (13,3.6) -- (7,3.6);

\fill[white] (8.4,2) circle (0.13) node[black, left] {$t$};
\fill[red] (8.5,2) circle (0.12) ;

\draw[ultra thick,black,opacity=0.3] (10,2.5) -- (11.6,4.1) node[at end, above] {$\iota_{b}$};

\draw[ultra thick,black] (10,3.6) -- (11.6,5.2) node[midway, above left] {$\iota_{b}$};

\fill[white] (9.9,3.5) node[black, below] {$\underline{a}$};
\fill[red] (10,3.6) circle (0.15);

\fill[white] (10,2.5) circle (0.14);
\fill[white] (10,2.4) circle (0.15) node[black,below right,opacity=0.3] {$S(\underline{a})$};

\fill[white] (10,3.6) circle (0.12);

\fill[red,opacity=0.3] (10,2.5) circle (0.12);

\fill[white] (10.8,3.1) circle (0.03) node[black,right,opacity=0.3] {$q$};
\fill[white] (10.7,3.2) circle (0.12) ;
\fill[red,opacity=0.3] (10.7,3.2) circle (0.12) ;

\draw[ultra thick,black] (7.7,0.4) -- (9,1.7) node[midway, below right] {$\iota_{a-1}$};

\end{tikzpicture}
\caption[below]{Left: If $t>\underline{a}$ and $S(t)$ is the $a$ from a $\Pi(a,b)$ pattern, so is $t$. Right:
If $t<S(\underline{a})$ and $S(t)$ is the $a$ from a $\Pi(a,b)$ pattern, so is $t$.}\label{fig9}
\end{figure}
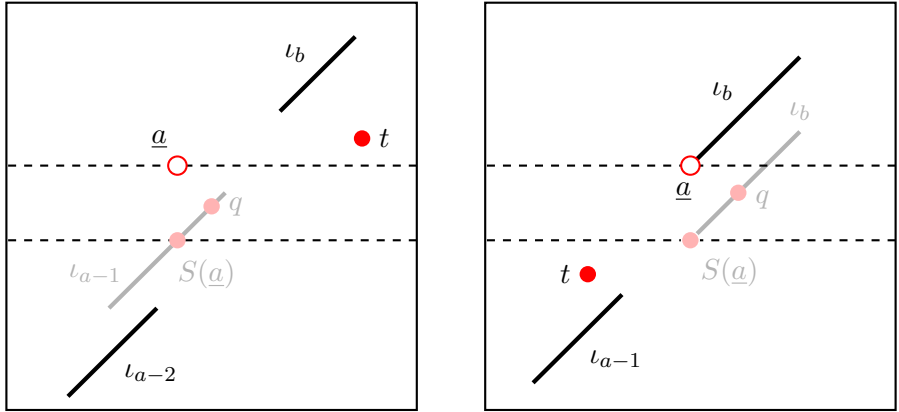

If $t<S(\underline{a})$ (see Figure \ref{fig9}, right) then there is some increasing subpermutation $\iota_{b}$ above $t$ involving $S(\underline{a})$ and $S(q)$ for some $q$ that was an $a-1$ associated with $\underline{a}$ in $\pi$. We know that $\underline{a}$ was part of an increasing subpermutation $\iota_{b}$ in $\pi$. Replacing the portion of the $\iota_{b}$ in $\rho$ that is above $S(\underline{a})$ in $S(\pi)$ with an equivalent sized portion of the $\iota_{b}$ in $\pi$ containing $\underline{a}$, we find an occurrence of $\rho$ in $\pi$ with $a$ element $t$, again contradicting our assumption on $t$. 
\end{proof}

Lemma \ref{Lownewj} implies that if we iteratively apply $S$ to a permutation $\pi\in \Av(\Pi(a,b))$, as each iteration introduces more patterns from $\Pi(a,b)$, the $a$ element in these introduced patterns in $S^k(\pi)$ is not greater than the $\underline{a}$ element of $S^{k-1}(\pi)$. We will now use this result to prove that $S^{n-a}$, restricted to the domain $\Av_n(\Pi(a,b))$, is injective.

\begin{theorem}\label{injective}
The map $S^{n-a}:\Av_n(\Pi(a,b)) \rightarrow \Av_n(\Pi(a-1,b+1))$ is injective.
\end{theorem}
\begin{proof}
Let $\pi,\tau\in \Av_n(\Pi(a,b))$ such that $\pi\neq \tau$ and suppose that $S^{n-a}(\pi) = S^{n-a}(\tau)$. By Corollary~\ref{cor: image of S n-a}, $S^{n-a}(\pi)$ and $S^{n-a}(\tau)$ are in $ \Av_n(\Pi(a-1,b+1))$. If $\pi$ and $\tau$ are both in $\Av_n(\Pi(a-1,b+1))$, then $\pi = S^{n-a}(\pi) = S^{n-a}(\tau) = \tau$, contradicting $\pi\neq \tau$. Similarly, suppose that $\pi\in \Av_n(\Pi(a-1,b+1))$ and $\tau \notin \Av_n(\Pi(a-1,b+1))$ (or vice versa). Then we have $\pi=S^{n-a}(\tau)$. This is a contradiction, since $\pi\in \Av_n(\Pi(a,b))$, and $S^{n-a}(\tau)$ contains patterns from $\Pi(a,b)$, as these are created at each iteration of the map $S$.

Suppose neither $\pi$ nor $\tau$ are in $\Av_n(\Pi(a-1,b+1))$. This means that there are some $0\leqslant  c,d< n-a$ such that 
\[S^c(\pi)\neq S^d(\tau)\text{ and }S^{c+1}(\pi) = S^{d+1}(\tau).\]

For simplicity we will write $S^c(\pi)$ as $\pi'$ and $S^d(\tau)$ as $\tau'$. Let $\underline{a}_{\pi'}$ and $\underline{a-1}_{\pi'}$ denote the $\underline{a}$ and $\underline{a-1}$ elements in $\pi'$, and let $\underline{a}_{\tau'}$ and $\underline{a-1}_{\tau'}$ denote the $\underline{a}$ and $\underline{a-1}$ elements in $\tau'$.

First we will show that $\underline{a}_{\pi'}$ and $\underline{a}_{\tau'}$ must have the same value. Suppose without loss of generality that $\underline{a}_{\pi'}$ is greater than $\underline{a}_{\tau'}$, as shown in Figure \ref{Case1}.

\begin{figure}[h]
\centering
\begin{tikzpicture}[scale=0.9]

\draw[thick] (0,0) rectangle (6,6);

\fill[white] (1.3,4) circle (0.12) node[black, above right]{$\underline{a}_{\pi'}$};

\fill[white] (12,3.4) circle (0.12) node[black, below right]{$\underline{a}_{\tau'}$};
\draw[thick](6,2.8)--(0,2.8);
\draw[thick] (6,4) -- (0,4);

\fill[white] (3,6.5) circle (0.01) node[black] {$\pi'$};
\fill[white] (10,6.5) circle (0.01) node[black] {$\tau'$};

\draw[thick, dashed](6,2.2)--(0,2.2);

\draw[thick] (7,0) rectangle (13,6);
\draw[thick](13,3.4)--(7,3.4);
\draw[thick](13,2.2)--(7,2.2);
\draw[thick,dashed] (13,4) -- (7,4);

\fill[red] (10,2.2) circle (0.12) node[black, above left]{$\underline{a-1}_{\tau'}$};

\fill[red] (3.5,2.8) circle (0.12) node[black, above right]{$\underline{a-1}_{\pi'}$};

\draw[ultra thick] (2.3,1.5) -- (3.3,2.5);
\fill[white] (3.3,1.6) circle (0.01) node[black]{$\iota_{a-2}$};

\draw[ultra thick] (3.7,4.2) -- (4.7,5.2);
\fill[white] (4,5) circle (0.01) node[black]{$\iota_{b}$};

\draw[ultra thick] (8.8,0.9) -- (9.9,2);
\fill[white] (9.9,1.5) circle (0.01) node[black, below]{$\iota_{a-2}$};

\draw[ultra thick] (10.2,3.6) -- (11.2,4.6);
\fill[white] (10.6,4.6) circle (0.01) node[black]{$\iota_{b}$};

\fill[red] (1.3,4) circle (0.12);
\fill[red] (12,3.4) circle (0.12);

\end{tikzpicture}
\caption[below]{The permutation diagrams of  $\pi'\neq \tau'$, supposing that $\underline{a}_{\pi'}>\underline{a}_{\tau'}$.}\label{Case1}
\end{figure}
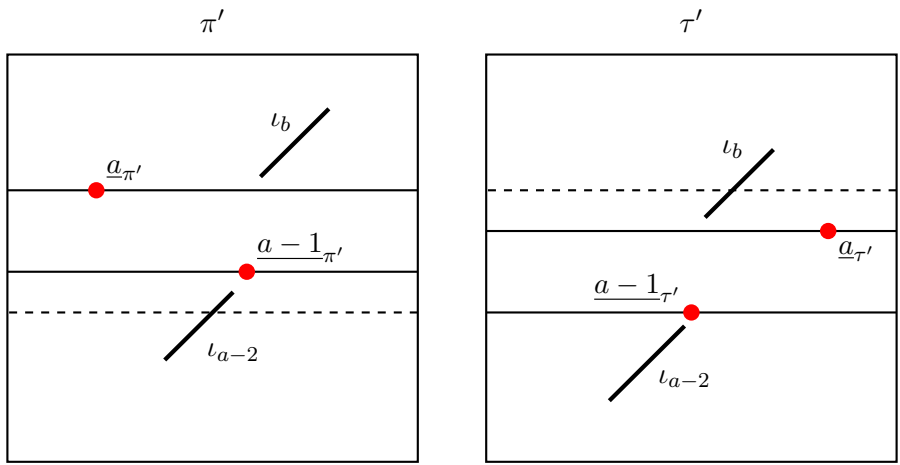

The map $S$ sends $a-1$ elements from patterns in $\Pi(a-1,b+1)$ (excluding $\sigma_{a,b}$) to $a$ elements in patterns from $\Pi(a,b)$. Therefore, $S$ introduces a pattern from $\Pi(a,b)$ in $S(\pi')$ that has an $a$ element with value $\underline{a}_{\pi'}>\underline{a}_{\tau'}$. By Lemma~\ref{Lownewj}, $S(\tau')$ cannot have a pattern from $\Pi(a,b)$ with $a$ element greater than $\underline{a}_{\tau'}$. This contradicts our assumption that $S(\pi') = S(\tau')$, so $\underline{a}_{\pi'}$ and $\underline{a}_{\tau'}$ must have the same value.

Suppose that $\underline{a}_{\pi'}$ and $\underline{a}_{\tau'}$ are not in the same position. Without loss of generality, suppose that $\underline{a}_{\pi'}$ occupies a position in $\pi'$ that is to the right of the position of $\underline{a}_{\tau'}$ in $\tau'$. Since $S(\pi')=S(\tau')$, there must be some element in $\tau'$ that has value less than $\underline{a}_{\tau'}$, and occupies the same position as that which $S(\underline{a}_\pi')$ holds in $S(\pi')$. Let us denote this element $\pi_{a}$. Since $S(\tau')=S(\pi')$, $\pi_a$ must fill essentially the same role in $\tau'$ as the image of $\underline{a}_{\pi'}$ does in $S(\pi')$. In particular, we know that $\pi_a$ is northeast of an $\iota_{a-2}$, and also that there is an $\iota_b$ northeast of $\pi_{a}$, that is also northeast of $\underline{a}_{\tau'}$ since $\underline{a}_{\tau'}$ and $\underline{a}_{\pi'}$ have the same value. As can be seen in Figure \ref{Case2}, this gives a pattern from $\Pi(a,b)$ with $a$ element $\underline{a}_{\tau'}$, contradicting Lemma \ref{Lownewj}.

Finally we consider when $\underline{a}_{\pi'}$ and $\underline{a}_{\tau'}$ have the same value and position in $\pi'$ and $\tau'$ respectively. Since $S(\pi')=S(\tau')$ and $\underline{a}_{\pi'}=\underline{a}_{\tau'}$, it must be the case that $S(\underline{a}_{\pi'})=S(\underline{a}_{\tau'})$, from which it follows that $\pi'=\tau'$, contradicting our initial assumption.\end{proof}

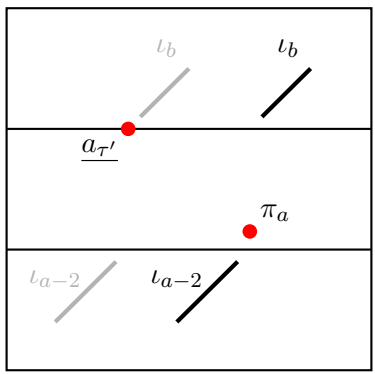
\begin{figure}[h]
\centering
\begin{tikzpicture}[scale=0.8]

\draw[thick] (0,0) rectangle (6,6);
\draw[thick](6,2)--(0,2);
\draw[thick] (6,4) -- (0,4);

\fill[red] (2,4) circle (0.12) node[black,below left]{$\underline{a_{\tau'}}$};

\draw[ultra thick,opacity=0.3] (2.2,4.2) -- (3,5) node[black, above left]{$\iota_b$};
\draw[ultra thick,opacity = 0.3] (0.8,0.8) -- (1.8,1.8);
\fill[white](0.8,1.5) node[black, opacity=0.3]{$\iota_{a-2}$};

\fill[red] (4,2.3) circle (0.12) node[black,above right]{$\pi_{a}$};

\draw[ultra thick] (4.2,4.2) -- (5,5) node[black, above left]{$\iota_b$};
\draw[ultra thick] (2.8,0.8) -- (3.8,1.8);
\fill[white](2.8,1.5) node[black]{$\iota_{a-2}$};

\end{tikzpicture}
\caption[below]{If $\underline{a}_{\tau'}$ and $ \underline{a}_{\pi'}$ have the same value but different positions, then the permutation $\tau'$ contains a pattern from $\Pi(a,b)$ with $a$ element $\underline{a}_{\tau'}$. }\label{Case2}
\end{figure}

In order to prove the Wilf-equivalence result central to this section, we need to describe some symmetries on permutations. The \emph{complement} of a permutation $\pi = \pi_1\ldots \pi_n$ is the permutation $\pi^c = (n+1-\pi_1)\ldots (n+1-\pi_n)$, the \emph{reverse} is $\pi^r = \pi_n\ldots \pi_1$, and the \emph{reverse complement} $\pi^{rc}$ is the reverse of the complement. It is straightforward to check that if $\sigma\in \Av(\Pi)$ then $\sigma^{rc}\in \Av(\Pi^{rc})$, where $\Pi^{rc} = \{\pi^{rc}:\pi\in \Pi\}$. 

We can now combine the results of this section to prove the Wilf-equivalence of partial shuffles of the same size, but first, we observe the following. When $S$ acts on $\pi\in \Av(\sigma_{a,b})$, it preserves the descent set, i.e., the set of indices $i$ such that $\pi_i>\pi_{i+1}$. 
To see this, note that if $\pi_i=\underline{a}$ then $\pi_{i-1}$ is never an associated $a-1$, and if $\pi\in \Av(\sigma_{a,b})$, then neither is $\pi_{i+1}$. Hence $S$ does not remove or add any descents. This observation, along with the following result, is equivalent to \cite[Lemma 4.4]{bloom2020revisiting}.

\begin{theorem}\label{Nearincreasingbasis}
For any pairs of integers $a,b$ and $c,d$ such that $a,c\geqslant 1$, $b,d \geqslant 0$ and $a+b=c+d\geqslant 2$, the sets of patterns $\Pi(a,b)$ and  $\Pi(c,d)$ are Wilf-equivalent. 
\end{theorem}

\begin{proof}
By Lemma \ref{terminates}, if $a\geqslant 2$, the map $S^{n-a}$ sends permutations in  $\Av_n(\Pi(a,b))$ to permutations in  $\Av_n(\Pi(a-1,b+1))$. By Theorem \ref{injective}, $S^{n-a}$ is injective. Note that $(\Pi(a,b))^{rc} = \Pi(b+1,a-1)$ and  $(\Pi(a-1,b+1))^{rc} = \Pi(b+2,a-2)$. Hence, if we consider the reverse complements of the permutation classes that $S$ maps between, we have an injection in the other direction, from $\Av_n(\Pi(b+1,a-1))$ to $\Av_n(\Pi(b,a))$. From this it follows that $S^{n-a}$ is in fact a bijection between $\Av_n(\Pi(a,b))$ and  $\Av_n(\Pi(a-1,b+1))$, so the sets of patterns $\Pi(a,b)$ and $\Pi(a-1,b+1)$ are Wilf-equivalent. The result follows. 
\end{proof}

\section{Including a decreasing pattern in the basis.}

In this section we establish a larger family of Wilf-equivalent sets of patterns. We show that the map $S$ preserves the size of the longest decreasing subpermutation in $\Av(\Pi(a,b))$ and so for integers $a,c\geqslant 1$, $b,d\geqslant 0$ such that $a+b=c+d\geqslant 2$, the sets of patterns $\{\Pi(a,b),\delta_m\}$ and $\{\Pi(c,d),\delta_m\}$ are Wilf-equivalent. We then discuss the enumeration of the class $\Av(\Pi(a,b),\delta_m)$. 

\begin{lemma}\label{dontgaindeltamplus1}
Suppose that $\pi \in \Av(\sigma_{a,b})$ and that the longest decreasing subpermutation of $\pi$ has size $m$. Then $S(\pi)$ does not contain $\delta_{m+1}$. 
\end{lemma}
\begin{proof}
Suppose that the longest decreasing subpermutation of $\pi\in \Av(\sigma_{a,b})$ is $\delta_m$, and that $S(\pi)$ contains $\delta_{m+1}$. By Corollary \ref{ignorej}, the pattern $\delta_{m+1}$ contains $S(\underline{a})$, and a nonempty decreasing pattern $\delta_d$ that contains the images of one or more $a-1$ elements associated with $\underline{a}$ in $\pi$. We let $\delta_c$ and $\delta_{m-c-d}$ denote the (possibly empty) parts of $\delta_{m+1}$ above $\underline{a}$ and below $\underline{a-1}$ in $\pi$, respectively (see Figure \ref{deltamplus1}, left).

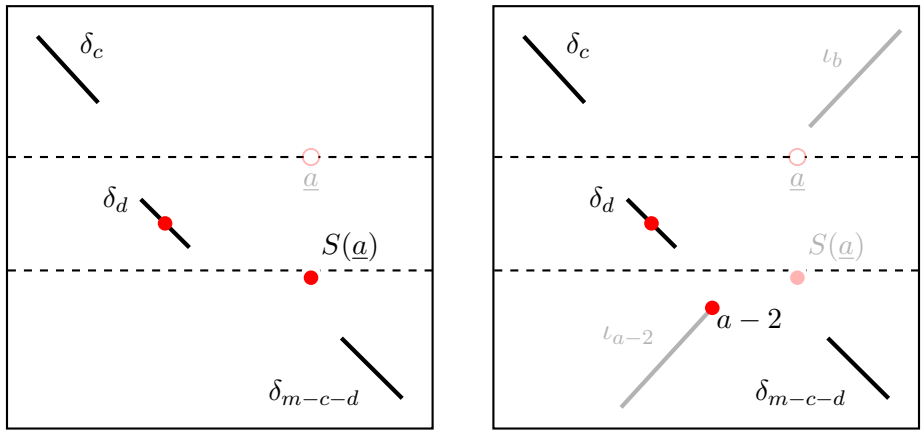
\begin{figure}[h]
\centering
\begin{tikzpicture}[scale=0.8]

\draw[thick] (0,0) rectangle (7,7);
\draw[thick,dashed](7,2.62)--(0,2.62);
\draw[thick,dashed] (7,4.5) -- (0,4.5);

\fill[white] (5,4.4) node[black, below,opacity = 0.3] {$\underline{a}$};
\fill[red,opacity = 0.3] (5,4.5) circle (0.15);
\fill[white] (5,4.5) circle (0.12);

\fill[white] (5,2.6) circle (0.15) node[black,above right] {$S(\underline{a})$};

\fill[red] (5,2.5) circle (0.12);

\draw[ultra thick, black] (2.2,3.8) -- (3,3) node[ at start, left] {$\delta_d$};

\draw[ultra thick, black] (0.5,6.5) -- (1.5,5.4) node[midway, above right] {$\delta_c$};

\draw[ultra thick, black] (5.5,1.5) -- (6.5,0.5) node[midway, below left] {$\delta_{m-c-d}$};

\fill[red] (2.6,3.4) circle (0.12) ;

\draw[thick] (8,0) rectangle (15,7);
\draw[thick,dashed](15,2.62)--(8,2.62);
\draw[thick,dashed] (15,4.5) -- (8,4.5);

\draw[ultra thick,black,opacity = 0.3] (13.2,5) -- (14.7,6.6) node[midway, above left] {$\iota_{b}$};

\fill[white] (13,4.4) node[black, below,opacity = 0.3] {$\underline{a}$};
\fill[red,opacity = 0.3] (13,4.5) circle (0.15);
\fill[white] (13,4.5) circle (0.12);

\fill[white] (13,2.6) circle (0.15) node[black,above right,opacity = 0.3] {$S(\underline{a})$};

\fill[red,opacity = 0.3] (13,2.5) circle (0.12);

\draw[ultra thick, black] (10.2,3.8) -- (11,3) node[at start, left] {$\delta_d$};

\draw[ultra thick, black] (8.5,6.5) -- (9.5,5.4) node[midway,above right] {$\delta_c$};

\draw[ultra thick, black] (13.5,1.5) -- (14.5,0.5) node[midway, below left] {$\delta_{m-c-d}$};

\fill[red] (10.6,3.4) circle (0.12) ;

\draw[ultra thick,black,opacity = 0.3] (10.1,0.35) -- (11.6,2) node[midway,above left] {$\iota_{a-2}$};

\fill[white] (11.5,2.15) circle (0.1) node[black, below right]{$a-2$};

\fill[red] (11.6,2) circle (0.12);

\end{tikzpicture}
\caption[below]{Left: If $\pi\in \Av(\delta_{m+1})$ and $S(\pi)\notin \Av(\delta_{m+1})$, any $\delta_{m+1}\in S(\pi)$ must involve $S(\underline{a})$ and some associated $a-1$. Right: An $a-2$ between $S(\underline{a})$ and the $a-1$ implies $\pi\notin \Av(\delta_{m+1})$.}\label{deltamplus1}
\end{figure}

The element $S(\underline{a})$ must be southeast of this $\delta_d$ in $S(\pi)$, since $\delta_d$ and $S(\underline{a})$ form a decreasing subpermutation, and $S(\underline{a})$ is less than the images of its associated $a-1$ elements in $S(\pi)$. Since $S(\pi)$ does not contain $\sigma_{a,b}$, we know there is some $a-2$ element southeast of every element of $\delta_d$ and southwest of $\underline{a}$. We also know that every element of $\delta_{m-c-d}$ is below this $a-2$ element, since if some element were not, it would be an $a-1$ element for $\underline{a}$, and would be above $S(\underline{a})$ in $S(\pi)$. From here we see that the occurrence of the pattern $\delta_{m+1}$ in $S(\pi)$ must have already existed in $\pi$, with the $a-2$ element replacing $S(\underline{a})$ (see Figure \ref{deltamplus1}, right). 
This contradicts our assumption on $\pi$.
\end{proof}

\begin{lemma}\label{keepdeltam}
Suppose that $\pi\in \Av_n(\Pi(a,b))$. Then the longest decreasing subpermutations of $\pi$ and $S^{n-a}(\pi)$ have the same size. 
\end{lemma}
\begin{proof}
We will show that for every $0\leqslant k < n-a$, if the longest decreasing permutation that appears as a pattern in $S^k(\pi)$ is $\delta_m$, then the longest decreasing permutation that appears as a pattern in $S^{k+1}(\pi)$ is also $\delta_m$. For ease of notation we write $S^k(\pi) = \pi'$. This is trivially true if $\pi'\in \Av_n(\Pi(a-1,b+1))$, so suppose that $\pi'$ contains some pattern from $\Pi(a-1,b+1)$. We will verify that $\delta_m$ is contained in $S(\pi')$. By Lemma \ref{ignorej}, we only need to consider when $\underline{a}$ and some associated $a-1$ element(s) are part of $\delta_m$ in $\pi'$. Let $q$ denote the maximal $a-1$ element associated with $\underline{a}$ that is part of $\delta_m$. Since $q<\underline{a}$ in $\pi'$, it must be the case that $q$ is to the right of $\underline{a}$.

Since $\pi'\in \Av(\sigma_{a,b})$, there must be an $\iota_b$ with least element $t$ which is northeast of $\underline{a}$ and northwest of $q$. Let $s$ denote the smallest element of $\delta_m$ that is larger than $\underline{a}$ (if such an element exists). If $t<s$ (or if $s$ does not exist) then $S(\pi)$ contains $\delta_m$, replacing $\underline{a}$ with $t$ (see Figure \ref{deltampic}, left). If $t>s$, then the $\iota_{a-2}$ southwest of $S(\underline{a})$, the $\iota_b$ containing $t$, and $S(\underline{a})$ form an increasing subpermutation of size $a+b-1$, and combined with $s$, this gives a pattern from $\Pi(a,b)$ with $a$ element $s$ that is greater than $\underline{a}$, contradicting Lemma \ref{Lownewj} (see Figure \ref{deltampic}, right). Hence $S(\pi')$ contains $\delta_m$. 
\end{proof}

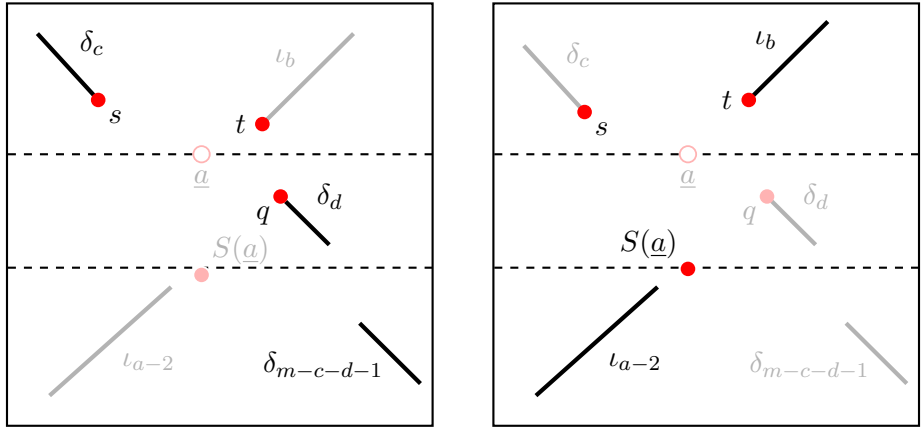
\begin{figure}[h]
\centering
\begin{tikzpicture}[scale=0.8]

\draw[thick] (0,0) rectangle (7,7);
\draw[thick,dashed] (7,2.62)--(0,2.62);
\draw[thick,dashed] (7,4.5) -- (0,4.5);

\draw[ultra thick,black,opacity = 0.3] (4.2,5) -- (5.7,6.5) node[midway, above left] {$\iota_{b}$};

\fill[white] (4.1,5) circle (0.12) node[black, left]{$t$};
\fill[red] (4.2,5) circle (0.12);

\fill[white] (3.2,4.4) node[black, below,opacity = 0.3] {$\underline{a}$};
\fill[red,opacity = 0.3] (3.2,4.5) circle (0.15);
\fill[white] (3.2,4.5) circle (0.12);

\fill[white] (3.2,2.5) circle (0.15) node[black,above right,opacity = 0.3] {$S(\underline{a})$};
\fill[red,opacity = 0.3] (3.2,2.5) circle (0.12);

\draw[ultra thick, black] (4.5,3.8) -- (5.3,3) node[midway, above right] {$\delta_d$};

\draw[ultra thick, black] (0.5,6.5) -- (1.5,5.4) node[midway,above right] {$\delta_c$};
\fill[red] (1.5,5.4) circle(0.12) node[black, below right]{$s$};

\draw[ultra thick, black] (5.8,1.7) -- (6.8,0.7);
\fill[white] (5.2,1) circle (0.01) node[black] {$\delta_{m-c-d-1}$};

\fill[white] (4.5,3.8) circle (0.12) node[black,below left]{$q$};
\fill[red] (4.5,3.8) circle (0.12);

\draw[ultra thick,black,opacity = 0.3] (0.7,0.5) -- (2.7,2.3) node[midway, below right] {$\iota_{a-2}$};

\draw[thick] (8,0) rectangle (15,7);
\draw[thick,dashed](15,2.62)--(8,2.62);
\draw[thick,dashed] (15,4.5) -- (8,4.5);

\draw[ultra thick,black] (12.2,5.4) -- (13.5,6.7) node[midway, above left] {$\iota_{b}$};

\fill[white] (12.1,5.4) circle (0.12) node[black, left]{$t$};
\fill[red] (12.2,5.4) circle (0.12) ;

\fill[white] (11.2,4.4) node[black, below,opacity = 0.3] {$\underline{a}$};
\fill[red,opacity = 0.3] (11.2,4.5) circle (0.15);
\fill[white] (11.2,4.5) circle (0.12);

\fill[white] (11.2,2.6) circle (0.15) node[black,above left] {$S(\underline{a})$};

\fill[red] (11.2,2.6) circle (0.12);

\draw[ultra thick, black, opacity=0.3] (12.5,3.8) -- (13.3,3) node[midway, above right] {$\delta_d$};

\draw[ultra thick, black, opacity = 0.3] (8.5,6.3) -- (9.5,5.2) node[midway,above right] {$\delta_c$};

\fill[red] (9.5,5.2) circle(0.12) node[black, below right]{$s$};

\fill[white] (13.2,1) circle (0.01) node[black,opacity=0.3] {$\delta_{m-c-d-1}$};
\draw[ultra thick, black,opacity=0.3] (13.8,1.7) -- (14.8,0.7) ;

\fill[white] (12.5,3.8) circle (0.12) node[black,opacity=0.3,below left]{$q$};

\fill[red,opacity=0.3] (12.5,3.8) circle (0.12) ;

\draw[ultra thick,black] (8.7,0.5) -- (10.7,2.3) node[midway, below right] {$\iota_{a-2}$};

\end{tikzpicture}
\caption[below]{Left: if $t<s$, the longest decreasing subpermutation $\delta_m$ is preserved by $S$. Right: if $t>s$, $S(\pi')$ contains a pattern from $\Pi(a,b)$ in which $s>\underline{a}$ acts as $a$, a contradiction.}\label{deltampic}
\end{figure}

Combining Lemmas \ref{dontgaindeltamplus1} and \ref{keepdeltam} with Theorem \ref{Nearincreasingbasis} proves the following:
\begin{theorem}\label{avabd}
For integers $a,b,c,d,m$ such that $a,c\geqslant 1$, $b,d\geqslant 0$, $a+b=c+d\geqslant 2$, and $m\geqslant 1$, the sets of patterns $\{\Pi(a,b), \delta_m\}$ and $\{\Pi(c,d),\delta_m\}$ are Wilf-equivalent.
\end{theorem}

We now consider the enumeration of permutations in $\Av(\Pi(a,b), \delta_m)$. 
A permutation class is called a \textit{polynomial class} if the number of permutations of size $n$ is counted by a polynomial in $n$, for sufficiently large $n$. That $\Av(\Pi(a,b), \delta_m)$ is a polynomial class is immediate from \cite[Theorem 1.3]{HV16}. 

Given a permutation $\sigma$ of size $m$, and monotone permutations $\alpha_1,\ldots,\alpha_m$, the \emph{inflation} of $\sigma$ by $\alpha_1,\ldots,\alpha_m$ is the permutation $\pi=\sigma[\alpha_1,\ldots,\alpha_m]$ obtained by replacing $\sigma_i$ with an interval of contiguous-valued elements determined by $\alpha_i$. For example, $3412[1,321,12,21]=5~876~12~43$. The polynomial classes are a specific type of \textit{geometric grid class} \cite{albert2013geometric} whose elements of a given size are inflations by monotone intervals of one of a finite number of permutations, that Homberger and Vatter christened \textit{peg permutations} \cite{HV16}. These are a convenient tool to study the enumeration of the classes $\Av_n(\Pi(a,b), \delta_m)$. 
A peg permutation $\tilde{\rho}$ has elements that are marked with $+,-,$ or ${\Cdot}$, indicating that they may be inflated with an arbitrarily long increasing permutation, an arbitrarily long decreasing permutation, or a single element, respectively. We can represent the permutation $3412[1,321,12,21]=58761243$ in this manner as $3^{\Cdot}4^-1^+2^-[1,3,2,2]$. The \textit{grid class} of a peg permutation $\tilde{\rho}$ is the set of permutations that may be obtained by inflating $\tilde{\rho}$ with monotone intervals as indicated by the markings on $\tilde{\rho}$, and a peg permutation is said to be \textit{compact} if its grid class is not wholly contained in the grid class of a different peg permutation. For example, $1^+2^{\Cdot}$ is not compact, since its inflations are contained in the grid class of $1^+$. 

In principle, one may enumerate a specific polynomial class by determining the maximum number of arbitrary inflations permitted within the set of compact peg permutations that define the class. At a slightly coarser level, we will use peg permutations to find the degree of the polynomials that count permutations in $\Av_n(\Pi(a,b),\delta_m)$, and in the case when $m=3$, we show that the leading coefficient of the counting polynomial is a Catalan number.

\begin{theorem}\label{polyord}
For any $a,b,m,$ such that $a\geqslant 1$, $a+b\geqslant 2$, $m\geqslant 2$, and for sufficiently large $n$, the number of permutations in $\Av_n(\Pi(a,b), \delta_m)$ is counted by a polynomial in $n$ with degree $(a+b-2)(m-2).$ 
\end{theorem}
\begin{proof}
We will prove that the theorem holds for  $\Av(\Pi(a+b,0), \delta_m)$, after which the general case follows from Theorem \ref{avabd}. 

The class  $\Av(\Pi(a+b,0), \delta_m)$ cannot contain the peg permutation $1^-$, so we need only consider inflation by increasing patterns. In order to avoid the pattern $(a+b)12\ldots (a+b-1)\in \Pi(a+b,0)$, the class cannot contain $21^+$, so inflation by an increasing permutation may only occur at left-to-right maxima. Furthermore, in order to avoid $\sigma_{a+b,0}$, the class cannot contain $1^+32$, meaning the only left-to-right maxima that may be marked with a $+$ are those that occur above every element that is not a left-to-right maximum. 
Since these left-to-right maxima are not dominated by any elements that are not themselves left-to-right maxima, we call them \textit{non-dominated left-to-right maxima}. We may bound the degree of the counting polynomial by bounding the number of non-dominated left-to-right maxima that a compact peg permutation $\tilde{\pi}$ whose expansions lie in $\Av(\Pi(a+b,0),\delta_m)$ may contain. Since $\tilde{\pi}$ is compact, every pair of non-dominated left-to-right maxima must have some lesser \textit{separating element} between them. In general, left-to-right maxima may be separated by value rather than position, but this creates a 132 pattern which cannot involve two non-dominated left-to-right maxima. The separating elements are below all of the non-dominated left-to-right maxima, and southeast of at least one of them. For $\tilde{\pi}$ to avoid $(a+b)12\ldots(a+b-1)$ and $\delta_m$, the separating elements cannot contain $\iota_{a+b-1}$ or $\delta_{m-1}$. Therefore, by the Erd\"{o}s-Szekeres theorem \cite{ES35}, there can be at most $(a+b-2)(m-2)$ separating elements. Since at least one separating element must occur between each pair of non-dominated left-to-right maxima, there can be at most $(a+b-2)(m-2)+1$ non-dominated left to right maxima. The number of ways that we may inflate such a peg permutation with $n$ elements is the number of ways to distribute $n$ elements into $(a+b-2)(m-2)+1$ slots, which is the coefficient of $x^n$ in the expansion of $(1-x)^{-((a+b-2)(m-2)+1)}$, given by a polynomial in $n$ of degree $(a+b-2)(m-2)$.

This places an upper bound on the degree of the counting polynomial. This upper bound is attained since the construction detailed in the previous paragraph is achieved by arranging the separating elements into any pattern from $\Av_{(a+b-2)(m-2)}(\iota_{a+b-1},\delta_{m-1})$, which is not empty, by the Erd\"{o}s-Szekeres theorem \cite{ES35}. 
\end{proof}

We demonstrate the construction used in the proof of Theorem \ref{polyord} with an example. 
\begin{example}
Suppose $a+b=4$ and $m=5$. The non-dominated left-to-right maxima form an increasing permutation of size 7. Beneath these, the separating elements may be arranged as any permutation from $\Av_6(\iota_3,\delta_4)$. In Figure \ref{pegperm} they form the pattern 563412. This gives the peg permutation $7^+ 5^{\Cdot} 8^+ 6^{\Cdot} 9^+ 3^{\Cdot} 10^+ 4^{\Cdot} 11^+ 1^{\Cdot} 12^+ 2^{\Cdot} 13^+$.

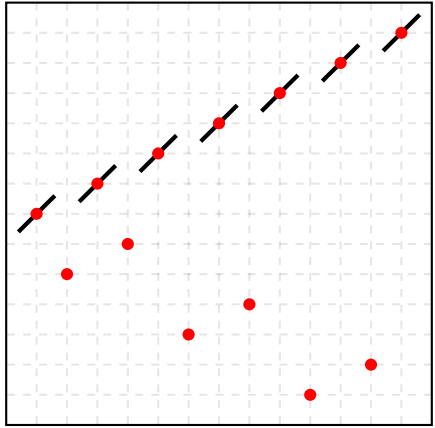
\begin{figure}[h]
\centering
\begin{tikzpicture}[scale=0.8]

\draw[thick] (0,0) rectangle (7,7);
\draw[thick,dashed,opacity = 0.1](7,0.5)--(0,0.5);
\draw[thick,dashed,opacity = 0.1](7,1)--(0,1);
\draw[thick,dashed,opacity = 0.1](7,1.5)--(0,1.5);
\draw[thick,dashed,opacity = 0.1](7,2)--(0,2);
\draw[thick,dashed,opacity = 0.1](7,2.5)--(0,2.5);
\draw[thick,dashed,opacity = 0.1](7,3)--(0,3);
\draw[thick,dashed,opacity = 0.1](7,3.5)--(0,3.5);
\draw[thick,dashed,opacity = 0.1](7,4)--(0,4);
\draw[thick,dashed,opacity = 0.1](7,4.5)--(0,4.5);
\draw[thick,dashed,opacity = 0.1](7,5)--(0,5);
\draw[thick,dashed,opacity = 0.1](7,5.5)--(0,5.5);
\draw[thick,dashed,opacity = 0.1](7,6)--(0,6);
\draw[thick,dashed,opacity = 0.1](7,6.5)--(0,6.5);
\draw[thick,dashed,opacity = 0.1](0.5,7)--(0.5,0);
\draw[thick,dashed,opacity = 0.1](1,7)--(1,0);
\draw[thick,dashed,opacity = 0.1](1.5,7)--(1.5,0);
\draw[thick,dashed,opacity = 0.1](2,7)--(2,0);
\draw[thick,dashed,opacity = 0.1](2.5,7)--(2.5,0);
\draw[thick,dashed,opacity = 0.1](3,7)--(3,0);
\draw[thick,dashed,opacity = 0.1](3.5,7)--(3.5,0);
\draw[thick,dashed,opacity = 0.1](4,7)--(4,0);
\draw[thick,dashed,opacity = 0.1](4.5,7)--(4.5,0);
\draw[thick,dashed,opacity = 0.1](5,7)--(5,0);
\draw[thick,dashed,opacity = 0.1](5.5,7)--(5.5,0);
\draw[thick,dashed,opacity = 0.1](6,7)--(6,0);
\draw[thick,dashed,opacity = 0.1](6.5,7)--(6.5,0);

\draw[ultra thick,black] (0.2,3.2) -- (0.8,3.8);
\fill[red] (0.5,3.5) circle (0.1) ;
\fill[red] (1,2.5) circle (0.1);
\draw[ultra thick,black] (1.2,3.7) -- (1.8,4.3);
\fill[red] (1.5,4) circle (0.1) ;
\fill[red] (2,3) circle (0.1) ;

\draw[ultra thick,black] (2.2,4.2) -- (2.8,4.8);
\fill[red] (2.5,4.5) circle (0.1);
\fill[red] (3,1.5) circle (0.1);
\fill[red] (4,2) circle (0.1);

\draw[ultra thick,black] (3.2,4.7) -- (3.8,5.3);
\fill[red] (3.5,5) circle (0.1);

\draw[ultra thick,black] (4.2,5.2) -- (4.8,5.8);
\fill[red] (4.5,5.5) circle (0.1);
\draw[ultra thick,black] (5.2,5.7) -- (5.8,6.3);
\fill[red] (5.5,6) circle (0.1);

\draw[ultra thick,black] (6.2,6.2) -- (6.8,6.8);
\fill[red] (6.5,6.5) circle (0.1);

\fill[red] (5,0.5) circle (0.1);
\fill[red] (6,1) circle (0.1);

\end{tikzpicture}
\caption[below]{A compact peg permutation for  $\Av(\Pi(4,0), \delta_5)$ that contains the maximum number of non-dominated left-to-right maxima, shown with an upward slanting line through them. }\label{pegperm}
\end{figure}

\end{example}

Next we will show that the leading coefficient of the polynomial that counts $\#\Av_n(\Pi(a,b),\delta_3)$ is a Catalan number. In order to do so, we first prove an auxiliary lemma that we were unable to find in the literature on Catalan numbers.

\begin{lemma}\label{Countinglemma}
The number of sequences of non-negative integers $a_1a_2\ldots a_{k-1}a_k$ such that \[\sum_{i\leqslant  j}a_i < j\]  for all $j\leqslant  k$ is given by $C_k$, the $k^{th}$ Catalan number. 
\end{lemma}

\begin{example}
There are $C_3=5$ such sequences of length 3: 000, 001, 010, 011, and 002.
\end{example}

\begin{proof}
According to \cite[Exercise~6.19(s)]{stanleyEC2}, the Catalan numbers count sequences $b_1b_2\ldots b_k$ such that $1\leqslant b_1\leqslant \cdots \leqslant b_k$, and $b_j\leqslant j$. We will refer to these as $b$-sequences, and those described in the statement of the lemma as $a$-sequences. Given an $a$-sequence, let $S_j$ denote the partial sum $\sum_{i=1}^j a_i$, for each $j\in [k]$. 
We claim the map $a_j \mapsto 1+S_j$ is a bijection between the $a$-sequences and the $b$-sequences. We have $1+S_j\leqslant j$ (since $S_j < j$), and $1+S_j \leqslant 1+S_{j+1}$, so the image of an $a$-sequence is a $b$-sequence.  Given a $b$-sequence $b_1b_2\ldots b_k$, the inverse map is given by $b_1 \mapsto 0$ and $b_j \mapsto b_j-b_{j-1}$ for $1<j\leqslant k$; note that this sends $b$-sequences to $a$-sequences since $b_j-b_{j-1}\geqslant 0$ and $\sum_{i\leqslant j} (b_i-b_{i-1}) = b_j-b_1 < j$. 
\end{proof}

\begin{theorem}
For integers $a,b$ such that $a\geqslant 1$, $b\geqslant 0$ and $a+b\geqslant 3$, and for sufficiently large $n$, the polynomial $p(n)$ that eventually counts $\#\Av_n(\Pi(a,b), \delta_3)$ has leading term $C_{a+b-2}{n\choose a+b-2}$.  
\end{theorem}

\begin{proof} 
We will prove that the theorem holds for  $\Av(\Pi(a+b,0), \delta_3)$, after which the general case follows from Theorem \ref{avabd}. Let $p(n)$ denote the polynomial that counts permutations in $\Av_n(\Pi(a+b,0), \delta_3)$ for sufficiently large $n$. 
By Theorem \ref{polyord} the degree of $p(n)$ is $a+b-2$, meaning that if $p(n)$ is expressed in the Newton basis for polynomials, $\left\{{n\choose k}\right\}_{k\geqslant 0}$, then $p(n)$ has leading term $c{n\choose a+b-2}$ for some $c$.

It remains to determine the coefficient $c$, and in order to do so, we count the compact peg permutations constructed as in Theorem \ref{polyord}, that contain $(a+b-2)(m-2)+1 = a+b-1$ non-dominated left-to-right maxima. Any compact peg permutation with fewer non-dominated left-to-right maxima will permit fewer arbitrary inflations, and will not contribute to the maximal degree term of $p(n)$. Let $\tilde{\pi}$ be some compact peg permutation for the class $\Av(\Pi(a,b), \delta_3)$, with $a+b-1$ non-dominated left-to-right maxima. Between each pair of non-dominated left-to-right maxima is a separating element, and since $\tilde{\pi}\in \Av(\delta_3)$, these $a+b-2$ separating elements must be increasing. Therefore, except for the values and positions of any \textit{dominated} left-to-right maxima, the structure of $\tilde{\pi}$ is fixed. Now we consider the number of ways we can select the dominated left-to-right maxima to construct $\tilde{\pi}$. 
 
The separating elements form the pattern $\iota_{a+b-2}$, and we cannot place a dominated left-to-right maxima in any position that creates an $\iota_{a+b-1}$. This means none may go southwest of the least separating element, at most one southwest of the next least separating element, at most two southwest of the next least separating element, and so on. We may place up to $a+b-3$ dominated left-to-right maxima southwest of the greatest separating element. These dominated left-to-right maxima must form an increasing subsequence, since they are all northwest of at least one separating element, and elements of the class must avoid $\delta_3$. 
 
Counting the total number of ways we can place elements in this manner amounts to enumerating the sequences \[\left\{a_1a_2\ldots a_{a+b-2}:0\leqslant  a_i \text{ and }\sum_{i\leqslant  j} a_i< j\right\}.\]

By Lemma \ref{Countinglemma}, this is the Catalan number $C_{a+b-2}$, meaning that the leading term of $p(n)$ is $C_{a+b-2}{n\choose a+b-2}$, as desired.\end{proof}

Automatic enumeration schemes (see \cite{HV16}) exist for polynomial classes such as these, so there is little to gain from further determining the exact coefficients of the polynomial that counts $\Av_n(\Pi(a,b), \delta_3)$. However, experimental data suggests the following form:
\begin{conjecture}
For $a+b\geqslant 3$ and $n\geqslant 2(a+b-2)+1$, the terms of the polynomial counting $\#\Av_n(\Pi(a,b), \delta_3)$ are given by \[C_{a+b-2}{n \choose a+b-2}- \sum_{1\leqslant  h< n-2} T_{a+b-2,h} {n \choose a+b-3-h},
\]
where the coefficients $T_{a+b-2,h}$ correspond to rows of the transposed Catalan triangle; \[T_{p,q} = \frac{q{2p-q \choose p}}{2p-q}\] (OEIS sequence A033184 \cite{oeis}).
\end{conjecture}
\begin{example}
For $n\geqslant 13$, \[\#\Av_n(\Pi(9,0), \delta_3) = 429 {n\choose 7} - 132 {n\choose 5}-132 {n\choose 4} - 90{n \choose 3}-48 {n\choose 2} - 20 {n\choose 1} - 6 {n\choose 0}. \]
\end{example}

\bibliographystyle{abbrv} 
\bibliography{Wilf-eq}

@article{albert2013geometric,
  title={Geometric grid classes of permutations},
  author={Albert, Michael and Atkinson, Michael and Bouvel, Mathilde and Ru{\v{s}}kuc, Nik and Vatter, Vincent},
  journal={Trans. Amer. Math. Soc.},
  volume={365},
  number={11},
  pages={5859--5881},
  year={2013}
}

@article{bloom2020revisiting,
  title={Revisiting pattern avoidance and quasisymmetric functions},
  author={Bloom, Jonathan S and Sagan, Bruce E},
  journal={Ann. Comb.},
  volume={24},
  number={2},
  pages={337--361},
  year={2020},
  publisher={Springer}
}

@article{ES35,
  title={A combinatorial problem in geometry},
  author={Erd{\"o}s, Paul and Szekeres, George},
  journal={Compos. Math.},
  volume={2},
  pages={463--470},
  year={1935}
}

@article{HPS20,
  title={Pattern avoidance and quasisymmetric functions},
  author={Hamaker, Zachary and Pawlowski, Brendan and Sagan, Bruce E},
  journal={Algebr. Comb.},
  volume={3},
  number={2},
  pages={365--388},
  year={2020}
}

@article{HV16,
  title={On the effective and automatic enumeration of polynomial permutation classes},
  author={Homberger, Cheyne and Vatter, Vincent},
  journal={J. Symbolic Comput.},
  volume={76},
  pages={84--96},
  year={2016},
  publisher={Elsevier}
}

@misc{oeis,
    Author = {{OEIS Foundation Inc.}},
    Note = {Published electronically at \url{http://oeis.org}},
    Title = {The {O}n-{L}ine {E}ncyclopedia of {I}nteger {S}equences},
    Year = 2025}

@book{stanleyEC2,
  author    = {Stanley, Richard P.},
  title     = {Enumerative Combinatorics. Vol. 2},
  series = {Cambridge Studies in Advanced Mathematics, Vol. 62},
  publisher = {Cambridge University Press},
  address = {Cambridge},
  year      = {1999},
  note = {With a foreword by Gian-Carlo Rota and appendix 1 by Sergey Fomin}
}
\end{document}